\documentclass[11pt]{amsart}
\usepackage{amscd,amssymb}
\usepackage{amsthm,amsmath}

\usepackage{fancyhdr}
\usepackage{supertabular}
\usepackage{setspace}
\usepackage{color}
\usepackage{longtable}

\usepackage[matrix,arrow,curve]{xy}
\usepackage{graphicx}
\usepackage{enumerate}
\usepackage{amsfonts}
\usepackage{cite}
\usepackage{pifont}
\usepackage{multirow}
\usepackage{hhline}
\usepackage[dvipsnames]{xcolor}

\usepackage{pgf,pgfarrows,pgfnodes,pgfautomata,pgfheaps,pgfshade}
\usepackage{tikz}

\makeatletter
\ifcase \@ptsize \relax
  \newcommand{\miniscule}{\@setfontsize\miniscule{3}{7}}
    \newcommand{\stiny}{\@setfontsize\miniscule{5}{7}}
\or
  \newcommand{\miniscule}{\@setfontsize\miniscule{3}{7}}
   \newcommand{\stiny}{\@setfontsize\miniscule{5}{7}}
\or
  \newcommand{\miniscule}{\@setfontsize\miniscule{3}{7}}
    \newcommand{\stiny}{\@setfontsize\miniscule{5}{7}}
\fi \makeatother

\newtheorem{theorem}[equation]{Theorem}
\newtheorem*{maintheorem*}{Main Theorem}
\newtheorem{lemma}[equation]{Lemma}
\newtheorem{proposition}[equation]{Proposition}

\newtheorem{conjecture}[equation]{Conjecture}

\newtheorem*{question*}{Question}

\newtheorem*{problem*}{Problem}

\theoremstyle{definition}

\newtheorem{definition}[equation]{Definition}

\newtheorem{Mtheorem}[equation]{Main Theorem}
\theoremstyle{remark}

\makeatletter\@addtoreset{equation}{section} \makeatother

\title{Alpha invariant for general polarizations of  del Pezzo surfaces of degree 1}

\author{Kyusik Hong and Joonyeong Won}

\address{\emph{Kyusik
Hong}\newline \textnormal{School of Mathematics, Korea Institute
for Advanced Study  \newline  85 Hoegiro, Dongdaemun-gu, Seoul 130-722,
Korea \newline \texttt{kszoo@kias.re.kr}}}%

\address{\emph{Joonyeong
Won}\newline \textnormal{Algebraic Structure and its Applications Research Center, KAIST \newline 335 Gwahangno, Yuseong-gu, Daejeon 305-701, Korea\newline \texttt{leonwon@kias.re.kr}}}%

\thanks{The second author has been supported by the National Research Foundation in Korea (NRF-2014R1A1A2056432)}
\begin{document}

\begin{abstract}
For an arbitrary ample divisor $A$ in smooth del Pezzo surface $S$ of degree 1, we verify
 the condition of the polarization $(S,A)$ to be K-stable and it is a simple numerical condition.
\end{abstract}

\maketitle

All considered varieties are assumed to be
algebraic and defined over an algebraically closed field of
characteristic $0$ throughout this article.
\section{Introduction}
The $\alpha$-invariant that is introduced by Tian \cite{Ti87} gives a numerical criterion
for the existence of K¨ahler-Einstein metrics on Fano manifolds.
The paper  \cite{Ti87}
proved that if $X$ is a Fano variety of dimension $n$ with canonical divisor $K_X$, the
lower bound $\alpha(X, −K_X) >\frac{ n}
{n+1}$ implies that $X$ admits a K\"{a}hler-Einstein metric in
$c_1(X) = c_1(−K_X)$.

 On the other hand, the Yau-Tian-Donaldson conjecture states that the existence of a constant scalar curvature K\"{a}hler  metric in
$c_1(A)$ for a polarised manifold $(X,A)$ is equivalent to the algebro-geometric notion
of K-stability,  a certain version of stability notion of Geometric Invariant Theory.
This conjecture has recently
been proven  when the divisor $A$ is anticanonical (\cite{cds1},\cite{cds2},\cite{cds3},\cite{Ti15}).
 Odaka and Sano
\cite{os12} have given a direct algebraic proof that $\alpha(X, -K_X)>
\frac{n}{n+1}$ implies that
$(X, −K_X)$ is K-stable. Then Dervan \cite{Dervan} generalizes this result, gives a sufficient condition for general polarisations of Fano varieties to be
K-stable.
\begin{theorem}\label{dervan}(\cite{Dervan})
Let $(X, A)$ be a polarised $\mathbb{Q}$-Gorenstein log canonical variety of dimension $n$ with
canonical divisor $K_X$. And let $\nu(A)=\frac{-K_X\cdot A^{n-1}}{A^n}$. Suppose that
\begin{itemize}
\item[(i)] $\alpha(X, A)>  \frac{n}
{n+1}\nu(A)$ and \\
\item[(ii)] $−K_X -\frac{ n}{n+1}\nu( A)A$ is nef.
\end{itemize}
Then $(X, A)$ is K-stable.

\end{theorem}

For anti-canonically
polarised del Pezzo surfaces, the computation is completely done by Cheltsov \cite{Ch07a}.  The result \cite{Ch07a} implies that general anticanonically polarized del pezzo surfaces of degree $\leq 3$ are K-stable. Generalizing this, Dervan checked K-stability for certain polarizations $(S,A_{\lambda})$, where $S$ is del Pezzo surface of degree 1 and $A_{\lambda}=-K_S+\lambda(\text{exceptional curve})$. The computation of $\alpha$-invariant is valuable in its own sake, and motivated by Dervan's results, we study the $\alpha$-invariant for all polarization of del Pezzo surfaces of degree 1. By the computation, it turns out that condition $(ii)$  is stronger than  condition $(i)$ in Theorem \ref{dervan} for del Pezzo surfaces of degree 1. In other words, the present article proves
\begin{Mtheorem}\label{mtheorem}
Let $(S,A)$ be a polarized del Pezzo surface of degree 1. Suppose that  $−K_S -\frac{ 2(-K_S\cdot A)}{3 (A)^2}A$ is nef. Then $(S,A )$ is K-stable.
\end{Mtheorem}

\section{Preliminaries and Notations}
\subsection{$\alpha$-invariant}For a polarized  smooth Fano variety $(X,A)$, its $\alpha$-invariant can be defined as
$$
\alpha(X,A)=\mathrm{sup}\left\{c\in\mathbb{Q}\ \left|%
\aligned
&\text{the log pair}\ \left(X, cD\right)\ \text{is log canonical for every}\\
&\text{effective $\mathbb{Q}$-divisor}\ D \sim_{\mathbb{Q}} A.\\
\endaligned\right.\right\}.%
$$
For every effective $\mathbb{Q}$-divisor $B$ on $X$, the number $$\mathrm{lct}(X,B)=\mathrm{sup}\left\{ c\in \mathbb{Q}\ |\ \text{the log pair}\ (X,cD)\ \text{is log canonical}\right\}$$
is called the $\emph{log canonical threshold}$ of $B$. Note that $$\alpha(X,A) =\mathrm{inf}\left\{\mathrm{lct}(X,B)|\ B\text{ is an effective $\mathbb{Q}$-divisor such that }B\sim_{\mathbb{Q}}A\right\}$$
Tian introduced $\alpha$-invariant of smooth Fano varieties in \cite{Ti87} and proved
\begin{theorem}
(\cite{Ti87} Theorem 2.1]). Let $X$ be a smooth Fano variety of dimension $n$. If
$\alpha(X,-K_X) > \frac{n}{n+1}$, then $X$ admits a K\"{a}hler-Einstein metric.
\end{theorem}
We will make use of $\alpha$-invariant for curves $\alpha_c$ to give a bound of the $\alpha$-invariant.
$$
\alpha_c(X,A)=\mathrm{sup}\left\{c\in\mathbb{Q}\ \left|%
\aligned
&\text{the log pair}\ \left(X, cD\right)\ \text{is log canonical along all curves }\\
&\text{for every effective $\mathbb{Q}$-divisor}\ D \sim_{\mathbb{Q}}A.\\
\endaligned\right.\right\}.%
$$
If the variety is a surface, then the number $\alpha_c(X,A)$ is a reciprocal of the maximal multiplicity along a curve of a divisor $B$, where $B$ is $\mathbb{Q}$-lineary equivalent to $A$.

The present article deals with a del Pezzo surface $S$ of degree 1 and  makes application of Theorem \ref{dervan}. So the slope $\nu(A)$ is always denoted by $\frac{-K_S\cdot A}{A^2}$

  \subsection{del Pezzo surfaces of degree 1}
Let $S$ be a smooth del Pezzo surface of degree~$1$. It can be obtained by blowing up $\mathbb{P}^2$ at eight points in general position. Let $\pi:S\to\mathbb{P}^2$ be such a blow up and $E_1,\ldots, E_8$ be
its exceptional curves. Denote the point $\pi(E_i)$ by $P_i$.

Let $h$ be the divisor class in $S$ corresponding to
$\pi^*(\mathcal{O}_{\mathbb{P}^2}(1))$ and $e_i$ be the class of
the exceptional curves $E_i$, where $i=1,\ldots, 8$. Since the classes
$h, e_1,\ldots, e_8$ form an orthogonal basis of the Picard group
of $S$, for a divisor $A$ on $S$ we may write $[A]=\beta h+\sum_{i=1}^8 \beta_ie_i$, where $\beta$ and $\beta_i$'s  are constants.
It is well known that the divisor $A$ is ample if and only if the intersection number $A\cdot C$ is positive for all $-1$-curves $C$ and the curve $C$ corresponds to one of following classes

\begin{itemize}
\item[$\bullet$] $e_i$ ;
\item[$\bullet$]  $h-e_i-e_j$   for  $i\neq j$ ;
\item[$\bullet$] $2h-e_i-e_j-e_k-e_l-e_m$ for $i\neq j\neq k\neq l\neq m$ ;
\item[$\bullet$] $[-K_S]+e_i-e_j$  for  $i\neq j$ ;
\item[$\bullet$] $[-2K_S]-(2h-e_i-e_j-e_k-e_l-e_m)$ for $i\neq j\neq k\neq l\neq m$ ;
\item[$\bullet$]  $[-2K_S]-(h-e_i-e_j)$   for  $i\neq j$ ;
\item[$\bullet$] $[-2K_S]-e_i$
\end{itemize}

In other words,  relations attained by the intersection number define the ample cone of $S$.
The Mori cone  $\overline{\mathbb{NE}}(S)$ of the surface $S$ is
polyhedral. Moreover,
 it
is generated by all the $(-1)$-curves on $S$.

From now on, the divisor $A$ is always assumed to be ample, unless otherwise stated.
The following method to express the divisor $A$ in terms of $-K_S$ and $(-1)$-curves is adopted from \cite{CheltsovParkWon3}, \cite{pw16}.
For the log pair $(S, A)$, we define an invariant of~$(S,A)$ by
$$
\mu:=\mathrm{inf}\left\{\lambda\in\mathbb{Q}_{>0}\ \Big|\ \text{the $\mathbb{Q}$-divisor}\ K_{S}+\lambda A\ \text{is pseudo-effective}\right\}.%
$$
The invariant $\mu$ is always attained by a positive rational number. Let  $\Delta_{(S,A)}$ be the smallest extremal face of the boundary of the Mori cone $\overline{\mathbb{NE}}(S)$ that contains $K_{S}+\mu A$.

 Let $\phi\colon S\to
Z$ be the contraction given by the face $\Delta_{(S,A)}$. Then either $\phi$ is a
birational  morphism or a conic bundle with
$Z\cong\mathbb{P}^1$. In the former case $\Delta_{(S,A)}$ is generated by $r$
disjoint $(-1)$-curves contracted by $\phi$, where  $r \leq 8$. In the later case,
$\Delta_{(S,A)}$ is generated by the $(-1)$-curves in the eight reducible fibers of $\phi$. Each reducible fiber
consists of two $(-1)$-curves that intersect transversally at one
point.

Suppose that $\phi$ is
birational. Let $E_1,\ldots,E_r $   be all $(-1)$-curves contained in
$\Delta_{(S,A)}$.  These are disjoint and generate the face $\Delta_{(S,A)}$. Therefore,
$$
K_{S}+\mu A\sim_{\mathbb{Q}}\sum_{i=1}^ra_iE_i
$$
for some positive rational numbers $a_1,\ldots,a_r$.
We have $a_i<1$ for every $i$ because  $A\cdot E_i>0$. Vice versa, for
every positive rational numbers $a_1,\ldots,a_r<1$, the divisor
$$
-K_{S}+\sum_{i=1}^ra_iE_i
$$
is ample.

 Suppose that $\phi$ is a conic
bundle. Then there are a $0$-curve $B$ and seven disjoint $(-1)$-curves $E_1, E_2, E_3, E_4,E_5,E_6,E_7$,  each of which is contained in a distinct fiber of $\phi$,
 such that
$$
K_{S}+\mu A\sim_{\mathbb{Q}} aB+\sum_{i=1}^7a_iE_i
$$
for some positive rational number $a$ and non-negative rational numbers $a_1, a_2, a_3 , a_4,a_5,a_6,a_7<1$.
In particular, these curves generate the face $\Delta_{(S,A)}$.
Vice versa, for  every positive rational number
$a$ and non-negative rational numbers $a_1, a_2, a_3 , a_4,a_5,a_6,a_7<1$
 the divisor
$$
-K_{S}+aB+\sum_{i=1}^7a_iE_i
$$
is ample.

Let us describe the notations we will use in the rest of the present paper. Unless otherwise
mentioned, these notations are fixed from now until the end of the paper.

\begin{itemize}
\item[$\bullet$] When the morphism $\phi$ is birational, $\mu A\sim_{\mathbb{Q}} -K_{S}+\sum_{i=1}^8a_iE_i $.\\
Fixing the order $a_1\geq a_2\geq\cdots \geq a_8$, $s_A$  = $\sum_{i=2}^8a_i$.
\item[$\bullet$]  When the morphism $\phi$ is conic bundle, $\mu A\sim_{\mathbb{Q}} -K_{S}+aB+\sum_{i=1}^7a_iE_i$.\\ Fixing the order $a_1\geq a_2\geq\cdots \geq a_7$, $s_A$  = $\sum_{i=2}^7a_i$.
\item[$\bullet$] $l_i$ is a $-1$-curve corresponding to a class $h-e_1-e_i$.
\item[$\bullet$] $Q$ is a $-1$-curve corresponding to a class $2h-e_1-e_5-e_6-e_7-e_8$.
\item[$\bullet$] $C_i$ is a $-1$-curve corresponding to a class $3h-2e_1-\sum_{j=2}^8e_j+e_i$.
\item[$\bullet$] $Z$ is a $-1$-curve corresponding to a class $6h-3e_1-2\sum_{j=2}^8e_j$.
\end{itemize}

\section{Log canonical thresholds  along curves}

  $$\mu A \sim_{\mathbb{Q}}-K_S +\sum_{i=1}^8 a_iE_i+aB,$$ where $a=0$ if $\phi $ is birational and $a_8=0$ if  $\phi $ is conic bundle.

Under the notations of Section 2, by choosing  six exceptional curves $D_1,\ldots, D_6$, where $\{D_1,\ldots, D_6\}\subset \{E_1,\ldots,E_6,E_7\}$, we obtain the birational morphism $S\rightarrow S_7$,   where $S_7$ is a del Pezzo surface of degree $7$. And there exist two disjoint $-1$-curves $D_7 $ and $D_8$ in $S_7$, We have the birational morphism $\pi : S\rightarrow \mathbb{P}^2$ defined by contraction of $D_1,\ldots, D_8$.  Let $d_1,\ldots,d_6,d_7,d_8$ be  divisor classes corresponding to $D_1,\ldots,D_8$ respectively.
If the morphism $\phi $ is birational, then  simply we have that $\{D_1,\ldots,D_8\}=\{E_1,\ldots,E_8\}$.

If the morphism $\phi $ is conic bundle and factors through $\mathbb{F}_1$, then we have that $\{D_1,\ldots,D_7\}=\{E_1,\ldots,E_7\}$. And  if the morphism $\phi $ is conic bundle and factors through $\mathbb{P}^1\times \mathbb{P}^1$, then we have that the divisor $D'=\{E_1,\ldots,E_7\}\setminus\{D_1,\ldots,D_6\}$ corresponds to a class $h-d_7-d_8$.

Note that the morphism $\pi$  depends on $E_i$'s.  We call $D_i$ as $\pi$-exceptional curve if it  belongs to the set that defines the morphism described previously.

For an effective divisor $D$ on a surface $X$, define a value $\sigma(D)$ to be $\mathrm{Max}\{a_i | D=\sum a_iD_i\text{, where $D_i$ is a irreducible curve}\}$.
\begin{definition}
On an algebraic surface $X$, we call the maximal multiplicity of divisor $A$ as $$\mathrm{sup}\{\sigma(D)~|~\text{$D$ is the effective $\mathbb{Q}$-divisor on $X$ and }D\sim_{\mathbb{Q}}A\}$$
\end{definition}

\begin{lemma}\label{mult_curve}
Suppose that  the maximal multiplicity $\alpha$ of $\mu A$ is greater than one.
Then either the maximal multiplicity is attained on an  $\pi$-exceptional curve $D_i$ or
 we have the following inequality
$$\frac{2+s_A+2a_1-a_7-a_8+3a}{3}\geq \alpha.$$
\end{lemma}

\begin{proof}
Note that any effective divisor on $S$ is generated by  $-1$-curves of 240 types and $K_1,K_2\in |-K_S|$ (cf. See \cite{tvv09}).  Suppose that an effective divisor $\alpha C + \Gamma$ is $\mathbb{Q}$-linearly equivalent to $\mu A$, where $C$ is an irreducible curve and the support of  $\Gamma$ does not contain $C$. The curve $C$ is linearly equivalent to $\sum b_iB_i$, where $b_i$ is integer and $B_i$ is $-1$-curve or $K_i$ and $b_i=1$ for all $i$ by the maximality of $\alpha$. . So we have that $\alpha(\sum B_i)+\Gamma\sim_{\mathbb{Q}}\mu A$.

Now consider the intersection
$$3=\mu A \cdot h\geq \alpha\sum B_k\cdot h,$$
where $h=\pi^*(\mathcal{O}_{\mathbb{P}^2}(1))$. It means that $B_k\cdot h\leq 2$. Thus $B_k$ is a $\pi$-exceptional curve or one of  curves $L_{ij}$ and $C_{ijlmn}$ which correspond to classes  $h-d_i-d_j$ and $2h-d_i-d_j-d_l-d_m-d_n$.

Now assume that $\alpha L_{ij}+\Omega \sim_{\mathbb{Q}}\mu A$, where  the support of  $\Omega$ does not contain $L_{ij}$. Let $C_{ip}$  and $D_j$ be a curve corresponding to a class $h-2d_p-\sum_{k\neq p}d_k +d_i$ and $d_j$, where  $p\notin \{i,j\}$.
In the cases that $\phi$ is birational or $\phi $ is conic bundle which factor through $\mathbb{F}_1$, then the following inequality holds
$$2+s_A+2a_1-a_7-a_8+3a\geq 2+s_A+2a_p-a_i-a_j+3a\geq (C_{ip}+D_j)\cdot \mu A\geq 3\alpha. $$
When  the morphism $\phi $ is conic bundle and factors through $\mathbb{P}^1\times \mathbb{P}^1$,
 Assume that $\{D_1,\ldots,D_6\}=\{E_1,\ldots,E_6\}$. If $\{i,j\}\subset\{1,\ldots,6\}$, then we have the same inequality
 $$2+s_A+2a_1-a_7-a_8+3a\geq 2+s_A+2a_p-a_i-a_j+3a\geq (C_{ip}+D_j)\cdot \mu A\geq 3\alpha. $$
 If $\{i,j\}\subset\{7,8\}$, that is, $L_{78}=E_7$, then contract $E_2,\ldots, E_6,E_7$, we have the morphism $S\rightarrow S_7$. And there is two disjoint $-1$-curves $D_7,D_8$ such that $E_1$ meets  $D_7,D_8$. The contraction $\pi'$ of $E_2,\ldots, E_6,E_7,D_7,D_8$ defines the morphism $S\rightarrow \mathbb{P}^2$. Thus $L_{78}$ is $\pi'$-exceptional.
 Now assume that $i=6,j=8$. then contract $E_1,\ldots, E_5,E_7$, we have the morphism $S\rightarrow S_7$. And there is two disjoint $-1$-curves $D_7,D_8$ such that $E_6$ meets  $D_7,D_8$.
  Moreover $L_{68}$ is one of $D_7,D_8$.
  The contraction $\pi''$ of $E_1,\ldots, E_5,E_7,D_7,D_8$ defines the morphism $S\rightarrow \mathbb{P}^2$. Thus $L_{68}$ is $\pi''$-exceptional

$\alpha C_{ijlmn}+\Omega \sim_{\mathbb{Q}}\mu A$, where  the support of  $\Omega$ does not contain $C_{ijlmn}$. Let $C_{ip}$  and $D_j$ be a curve corresponding to a class $h-2d_p-\sum_{k\neq p}d_k +d_i$ and $d_j$, where  $p\notin \{i,j,l,m,n\}$. Then in any case the following inequality holds
$$2+s_A+2a_1-a_7-a_8+3a\geq 2+s_A+2a_p-a_i-a_j+3a\geq (C_{ip}+D_j)\cdot \mu A\geq 3\alpha. $$

\end{proof}

In all lemmas of the present section,  it is easily verified that the maximal multiplicity is attained on a exceptional curve by Lemma \ref{mult_curve}. In other words, we can always find
a divisor whose multiplicity along $E_i$ is greater than or equal to the value of the bound in Lemma \ref{mult_curve}. Moreover, for all $\pi$-exceptional curves, the processes to find out maximal multiplicity along the curve are the same. Thus we will consider only maximal multiplicity along single exceptional curve $E_1$ which computes $\alpha_c(S,A)$ according to the order of $a_i$'s.

\subsection{birational morphism case.}

$$
K_{S}+\mu A\sim_{\mathbb{Q}}\sum_{i=1}^8a_iE_i
$$
for some positive rational numbers $a_1,\ldots,a_8$.

\begin{lemma}\label{s>4}
If $s_A\geq 4$, then $\alpha_c(S,\mu A)=\frac{1}{2+a_1}$.
\end{lemma}
\begin{proof}

There exist an effective divisor $$\mu A\sim_{\mathbb{Q}}(1-a_2)l_2+\cdots +
(1-a_7)l_7+(s_A-a_8-3)l_8+(2+a_1)E_1+(s_A-4)E_8.$$ Therefore we have that $\alpha_c(S,\mu A)\leq\frac{1}{2+a_1}$.

Now for $\eta<\frac{1}{2+a_1}$, suppose that the  pair $(S,\eta D)$ is not log canonical a along irreducible curve $C$, where the effective divisor $D$ is $\mathbb{Q}$-linearly equivalent to $\mu A$. Then we write $D=\alpha C+\Omega$, where the support of $\Omega$ does not contain  $C$. Since the inequality $\frac{2+s_A+2a_1-a_7-a_8+3a}{3}\leq 2+a_1$ holds, the curve $C$ is one of $E_i$ by Lemma \ref{mult_curve}.

Write $\mu A=\alpha E_i+\sum_{h\neq i} b_h E_h+\Omega$, where  the support of $\Omega$ does not contain $E_i$ and $E_h$'s. Then we have $$14+7a_i=(\sum_{p\neq i}^8 L_p+\sum_{h\neq i}^8 E_h)\cdot\mu A\geq 7\alpha,$$ where the $-1$-curve $L_p$ corresponds to a class of $h-e_p-e_i$,  so that $2+a_i\geq \alpha$. It is a contradiction, thus $\alpha_c(S,\mu A)= \frac{1}{2+a_1}$.



\end{proof}
For each $E_i$, considering  the  maximal multiplicity  along $E_i$ is similar and we easily have that  the maximum is attained on $E_1$ among them.
From now on we only consider the  maximal multiplicity along $E_1$ of $\mu A$.

\begin{lemma}\label{s23}
Assume that $1\leq s_A\leq 4$ and $1+a_4< a_2+a_3$. If $\frac{2}{3}+\frac{1}{3}s_A<a_2+a_3$, then $\alpha_c(S,\mu A)=\frac{2}{2+2a_1+s_A-a_2-a_3}$.
\end{lemma}

\begin{proof}
There is an effective divisor

\begin{eqnarray*}
\mu A\sim_{\mathbb{Q}} (1-a_2)C_3+(1-a_3)C_2+\frac{3a_2+3a_3-s_A-2}{2}Q_4+(a_2+a_3-a_4-1)l_4\\
+\sum_{i=5}^8\frac{s_A-2a_i-a_2-a_3}{2}l_i
+\frac{2+2a_1+s_A-a_2-a_3}{2}E_1.
\end{eqnarray*}
Therefore we have that  $\alpha_c(S,\mu A)\leq\frac{2}{2+2a_1+s_A-a_2-a_3}$.

Now for $\eta<\frac{2}{2+2a_1+s_A-a_2-a_3}$, suppose that the  pair $(S,\eta D)$ is not log canonical along a irreducible curve $C$, where the effective divisor $D$ is $\mathbb{Q}$-linearly equivalent to $\mu A$. Then we write $D=\alpha C+\Omega$, where the support of $\Omega$ does not contain  $C$. Since the inequality $\frac{2+s_A+2a_1-a_7-a_8+3a}{3}\leq \frac{2+2a_1+s_A-a_2-a_3}{2}$ holds, the curve $C$ is one of $E_1$ by Lemma \ref{mult_curve}.

If we write an effective divisor which is $\mathbb{Q}$-linearly equivalent to $\mu A$ as $D=\alpha E_1+b_1C_2+b_2C_3+c_1E_2+c_2E_3+\Omega$, where the support of $\Omega$ does not contain  $E_1,E_2,E_3,C_2$ and $C_3$, then the inequality
\begin{eqnarray*}
4+4a_1+2s_A-2a_2-2a_3=(C_2+C_3+E_2+E_3)\cdot D\geq 4\alpha.
\end{eqnarray*}
holds. So we obtain that $\alpha_c(S,\mu A)=\frac{2}{2+2a_1+s_A-a_2-a_3}$.
\end{proof}

\begin{lemma}\label{s234}
Assume that $1\leq s_A\leq 4$, $1+a_4\geq a_2+a_3$ and $1+2a_5<a_2+a_3+a_4$. If $\frac{1}{3}+\frac{2}{3}s_A<a_2+a_3+a_4$, then $\alpha_c(S,\mu A)=\frac{4}{3+4a_1+2s_A-a_2-a_3-a_4}$.
\end{lemma}

\begin{proof}
We can find an effective divisor

\begin{eqnarray*}
\mu A\sim_{\mathbb{Q}} \frac{1-a_3-a_4+a_2}{2}C_2 +\frac{1-a_2-a_4+a_3}{2}C_3+\frac{1-a_2-a_3+a_4}{2}C_4\\
+\frac{3a_2+3a_3+3a_4-2s_A-1}{4}Q
+\sum_{i=5}^8\frac{2s_A-a_2-a_3-a_4-4a_i-1}{4}l_i\\
+\frac{3+4a_1+2s_A-a_2-a_3-a_4}{4}E_1.
\end{eqnarray*}

Thus $\alpha_c(S,\mu A)\leq\frac{4}{3+4a_1+2s_A-a_2-a_3-a_4}$.

Now for $\eta<\frac{4}{3+4a_1+2s_A-a_2-a_3-a_4}$, suppose that the  pair $(S,\eta D)$ is not log canonical along a irreducible curve $C$, where the effective divisor $D$ is $\mathbb{Q}$-linearly equivalent to $\mu A$. Then we write $D=\alpha C+\Omega$, where the support of $\Omega$ does not contain  $C$. Since the inequality $\frac{2+s_A+2a_1-a_7-a_8+3a}{3}\leq \frac{3+4a_1+2s_A-a_2-a_3-a_4}{4}$ holds, the curve $C$ is one of $E_1$ by Lemma \ref{mult_curve}.

If we write
$$D=\alpha E_1+\sum_{i=5}^8 a_il_i + bQ+c_1C_2+c_2C_3+c_3C_4+\Omega,$$ where the support of $\Omega$ does not contain  $E_1,l_5,l_6,l_7,l_8,Q,C_2,C_3$ and $C_4$. Then the inequality

$$9+12a_1+6s_A-3a_2-3a_3-3a_4=(-K_S+Q+l_5+l_6+l_7+l_8+C_2+C_3+C_4)\cdot D\geq 12\alpha$$
holds. But it is absurd.
\end{proof}

\begin{lemma}\label{sR}
Assume that $1\leq s_A\leq 4$. If the inequality $1+2a_5\geq a_2+a_3+a_4$ holds or both $1+2a_5< a_2+a_3+a_4$ and $\frac{2}{3}+\frac{1}{3}s_A\geq a_2+a_3+a_4$ are satisfied, then $\alpha_c(S,\mu A)=\frac{3}{2+3a_1+s_A}$.
\end{lemma}
\begin{proof}
Suppose that inequalities $1+a_4\geq a_2+a_3$, $1+2a_5< a_2+a_3+a_4$,  and $\frac{2}{3}+\frac{1}{3}s_A\geq a_2+a_3+a_4$ hold. Then consider a divisor $D$ such that
\begin{eqnarray*}
D=\frac{1-a_3-a_4+a_2}{2}C_2 +\frac{1-a_2-a_4+a_3}{2}C_3+\frac{1-a_2-a_3+a_4}{2}C_4\\
+\sum_{i=5}^8\frac{a_2+a_3+a_4-2a_i-1}{2}l_i.\\
\end{eqnarray*}
By our assumption, sum of coefficients of $C_i$'s, $\frac{3-a_2-a_3-a_4}{2}$ is bigger than $\frac{4-s_A}{3}$,
there exists numbers $b_2,\ldots,b_4, c_2,\ldots,c_8$ satisfying following conditions
\begin{itemize}
\item[$\bullet$] all coefficients of the curves $C_i$'s in $D-b_2C_2-b_3C_3-b_4C_4$ are nonnegative,

\item[$\bullet$] $b_2+b_3+b_4=\frac{3-a_2-a_3-a_4}{2}-\frac{4-s_A}{3}$,
\item[$\bullet$] $\sum_{i=2}^8 c_i=6(b_2+b_3+b_4)$.
\end{itemize}
Then we have
\begin{eqnarray*}
\mu A\sim_{\mathbb{Q}}  D-b_2C_2-b_3C_3-b_4C_4+\sum_{i=2}^8 c_il_i+\frac{2+3a_1+s_A}{3}E_1
\end{eqnarray*}
 and all the coefficients of the divisor are nonnegative.

If inequalities $1+3a_6< a_2+a_3+a_4+a_5$ and $1+2a_5\geq a_2+a_3+a_4$ hold,
consider a divisor $D$ such that
\begin{eqnarray*}
D=\frac{1-a_3-a_4-a_5+2a_2}{3}C_2+\frac{1-a_2-a_4-a_5+2a_3}{3}C_3+\frac{1-a_2-a_3-a_5+2a_4}{3}C_4\\+\frac{1-a_2-a_3-a_4+2a_5}{3}C_5
+\sum_{i=6}^8 \frac{a_2+a_3+a_4+a_5-3a_i-1}{3}l_i.
\end{eqnarray*}
Since sum of coefficients of $C_i$'s, $\frac{4-a_2-a_3-a_4-a_5}{3}$ is bigger than $\frac{4-s_A}{3}$,
there exists numbers $b_2,\ldots,b_5, c_2,\ldots,c_8$ satisfying following conditions
\begin{itemize}
\item[$\bullet$] all coefficients of the curves $C_i$'s in $D-b_2C_2-b_3C_3-b_4C_4-b_5C_5$ are nonnegative,

\item[$\bullet$] $b_2+b_3+b_4+b_5=\frac{4-a_2-a_3-a_4-a_5}{3}-\frac{4-s_A}{3}$,
\item[$\bullet$] $\sum_{i=2}^8 c_i=6(b_2+b_3+b_4+b_5)$.
\end{itemize}
Then we have
\begin{eqnarray*}
\mu A\sim_{\mathbb{Q}}  D-b_2C_2-b_3C_3-b_4C_4-b_5C_5+\sum_{i=2}^8 c_il_i+\frac{2+3a_1+s_A}{3}E_1
\end{eqnarray*}
 and all the coefficients of the divisor are nonnegative.

If inequalities $1+4a_7< a_2+a_3+a_4+a_5+a_6$ and $1+3a_6\geq a_2+a_3+a_4+a_5$ hold, then consider a divisor $D$
\begin{eqnarray*}
D=\frac{1-a_3-a_4-a_5-a_6+3a_2}{4}C_2+\frac{1-a_2-a_4-a_5-a_6+3a_3}{4}C_3\\+\frac{1-a_2-a_3-a_5-a_6+3a_4}{4}C_4
+\frac{1-a_3-a_4-a_5-a_6+3a_5}{4}C_5\\+\frac{1-a_2-a_3-a_4-a_5+3a_6}{4}C_6
+\sum_{i=7}^8 \frac{a_2+a_3+a_4+a_5+a_6-4a_i-1}{4}l_i.
\end{eqnarray*}

The sum of coefficients of $C_i$'s, $\frac{5-a_2-a_3-a_4-a_5-a_6}{4}$ is bigger than $\frac{4-s_A}{3}$ provided by $s_A\geq 1$.
In same manner as a previous case, there is a coefficient  $b_2,\ldots,b_6,c_2,\ldots,c_8$    satisfying
\begin{eqnarray*}
\mu A\sim_{\mathbb{Q}}  D-b_2C_2-b_3C_3-b_4C_4-b_5C_5-b_6C_6+\sum_{i=2}^8 c_il_i+\frac{2+3a_1+s_A}{3}E_1.
\end{eqnarray*}

If both inequalities $1+5a_8< a_2+a_3+a_4+a_5+a_6+a_7 $ and $1+4a_7< a_2+a_3+a_4+a_5+a_6$ hold, then consider a divisor $D$
\begin{eqnarray*}
D=\frac{1-a_3-a_4-a_5-a_6-a_7+4a_2}{5}C_2+\frac{1-a_2-a_4-a_5-a_6-a_7+4a_3}{5}C_3\\+\frac{1-a_2-a_3-a_5-a_6-a_7+4a_4}{5}C_4
+\frac{1-a_3-a_4-a_5-a_6-a_7+4a_5}{5}C_5\\+\frac{1-a_2-a_3-a_4-a_5-a_7+4a_6}{5}C_6+\frac{1-a_2-a_3-a_4-a_5-a_6+4a_7}{5}C_7
\\+ \frac{a_2+a_3+a_4+a_5+a_6+a_7-5a_8-1}{5}l_8.
\end{eqnarray*}
In same manner as a previous case, there are coefficients  $b_2,\ldots,b_7,c_2\ldots,c_8$    satisfying
\begin{eqnarray*}
\mu A\sim_{\mathbb{Q}}  D-b_2C_2-b_3C_3-b_4C_4-b_5C_5-b_6C_6-b_7C_7+\sum_{i=2}^8 c_il_i+\frac{2+3a_1+s_A}{3}E_1.
\end{eqnarray*}

In remaining cases, $1+5a_8\geq a_2+a_3+a_4+a_5+a_6+a_7 $, consider an effective divisor $D=\sum_{i=2}^8\frac{1-s_A+6a_i}{6}C_i$. By choosing suitable choice of $b_2,\ldots,b_8,c_2,\ldots,c_8$. We have an effective divisor
 \begin{eqnarray*}
\mu A\sim_{\mathbb{Q}}  D-\sum_{i=2}^8b_iC_i+\sum_{j=2}^8 c_jl_j+\frac{2+3a_1+s_A}{3}E_1.
\end{eqnarray*}

Therefore we obtain that $\alpha_c(S,\mu A)\leq\frac{3}{2+3a_1+s_A}$.

Now for $\eta<\frac{3}{2+3a_1+s_A}$, suppose that the  pair $(S,\eta D)$ is not log canonical along a irreducible curve $C$, where the effective divisor $D$ is $\mathbb{Q}$-linearly equivalent to $\mu A$. Then we write $D=\alpha C+\Omega$, where the support of $\Omega$ does not contain  $C$. Since the inequality $\frac{2+s_A+2a_1-a_7-a_8+3a}{3}\leq \frac{2+3a_1+s_A}{3}$ holds, the curve $C$ is one of $E_1$ by Lemma \ref{mult_curve}.

Write $$D=\alpha E_1+al_2 +bC_2+cZ+\Omega,$$ where the support of $\Omega$ does not contain $E_1,l_2,C_2$ and $Z$. Then an inequality

$$4+6a_1+3s_A=(l_2+C_2+Z)\cdot D\geq 6\alpha$$
holds. But it is a contradiction so that $\alpha_c(S,\mu A)=\frac{3}{2+3a_1+s_A}$.

\end{proof}

\begin{lemma}
 If  $ s_A\leq 1$, then $\alpha_c(S,\mu A)=\mathrm{min}\{\frac{2}{1+2a_1+s_A},1\}$.
\end{lemma}
\begin{proof}
There exists an effective divisor

\begin{eqnarray*}
\mu A\sim_{\mathbb{Q}}  \frac{1-s_A}{2}Z+\sum_{i=2}^8a_iC_i+\frac{1+2a_1+s_A}{2}E_1.
\end{eqnarray*}
Therefore we obtain that $\alpha_c(S,\mu A)\leq\mathrm{min}\{\frac{2}{1+2a_1+s_A},1\}$.

For $\eta<\frac{2}{1+2a_1+s_A}$, suppose that the  pair $(S,\eta D)$ is not log canonical along a irreducible curve $C$, where the effective divisor $D$ is $\mathbb{Q}$-linearly equivalent to $\mu A$. Then we write $D=\alpha C+\Omega$, where the support of $\Omega$ does not contain  $C$. Since the inequality $\frac{2+s_A+2a_1-a_7-a_8+3a}{3}\leq \frac{1+2a_1+s_A}{2}$ holds, the curve $C$ is one of $E_1$ by Lemma \ref{mult_curve}.

If we write an effective divisor $D$ as
$D=\alpha E_1+bZ+\sum_{i=2}^8c_iC_i+\Omega$, where the support of $\Omega$ does not contain $E_1, C_2,\ldots,C_8$ and $Z$, then we have the following inequality
\begin{eqnarray*}
9+18a_1+9s_A=(-K_S+Z+\sum_{i=2}^8C_i)\cdot D\geq 18\alpha.
\end{eqnarray*}

Thus $\alpha_c(S,\mu A)=\mathrm{min}\{\frac{2}{1+2a_1+s_A},1\}$.

\end{proof}

Stating with previous lemmas together, we have the following proposition.

\begin{proposition}\label{alphaCbirational}
Let $S$ be a smooth del Pezzo surface and $A$ be a ample divisor of $S$.  If the morphism $\phi$ is birational, $\mu A\sim_{\mathbb{Q}} -K_{S}+\sum_{i=1}^8a_iE_i $, then

\begin{itemize}
\item[$\bullet$]When $s_A>4$,  $\alpha_c(S,A)=\frac{1}{2+a_1}$;
\item[$\bullet$]When $4\geq s_A>1$, $\alpha_c(S,A)=\mathrm{Max}\{\frac{2}{2+2a_1+s_A-a_2-a_3},\frac{4}{3+4a_1+2s_A-a_2-a_3-a_4},\frac{3}{2+3a_1+s_A}\}$;
\item[$\bullet$]When $1\geq s_A$,  $\alpha_c(S,A)= \mathrm{min}\{\frac{2}{1+2a_1+s_A},1\}$.
\end{itemize}


\end{proposition}

\subsection{Conic bundle case} Suppose that the contraction
 $\phi:S\to Z$ given by the face $\Delta_{(S,A)}$ is a conic
bundle, i.e., $Z=\mathbb{P}^1$.  The face $\Delta_{(S,A)}$ is spanned by an irreducible fiber $B$ of $\phi$ and $8$ disjoint $(-1)$-curves
$E_1 \ldots,E_7$.
We may then write
$$
K_{S}+\mu A\sim_{\mathbb{Q}} aB+\sum_{i=1}^{7}a_iE_i,
$$
where $a$ is a positive
rational number and $a_i$'s are non-negative rational numbers.
And we can assume that $a_1\geq  \cdots \geq a_7$.
 Let $\phi_1:
S\to R$ be the birational morphism obtained by contracting the
disjoint $(-1)$-curves $E_1,\ldots, E_7$.
And we can assume that $a_1\geq  \cdots \geq a_7$.
\medskip

\textbf{Subcase 1:}  $R$ is isomorphic to the Hirzebruch surface $\mathbb{F}_1$.

In this subcase, we have an extra $(-1)$-curve $E_8$ which
is disjoint from $E_1,\ldots,E_7$.
Assume that $s_A\geq 4$. Since $\frac{2+s_A+2a_1-a_7-a_8+3a}{3}\leq 2+a_1+a$, by the similar way of Lemma $\ref{s>4}$, we obtain that the multiplicity of $\mu A$ along $E_1$ has upper bound  $2+a_1+a$,
since if we write $\mu A=\alpha E_1+\sum_{k=2}^7 E_k+\Omega$, where the support of $\Omega$ does not contain $E_1,\ldots,E_7$, we have that $$14+7a_1+7a=(\sum_{i=2}^8l_i+\sum_{j=2}^8E_j)\cdot \mu A \geq 7\alpha.$$
There exist an effective divisor  $$\mu A\sim_{\mathbb{Q}}(1-a_2)l_2+\cdots +
(1-a_7)l_7+(s_A-3+a)l_8+(2+a_1+a)E_1+(s_A-4)E_8,$$ which is a sum of divisors  $a(l_8+E_8)$ and an effective divisor obtained in Lemma \ref{s>4}. Thus $\alpha_c(S,A)= \frac{1}{2+a_1+a}$.

  When $1\leq s_A \leq 4$.    Likewise we obtain upper bounds,  $\mathrm{mult}_{E_1}(\mu A)\leq\frac{2+2a_1+s_A-a_2-a_3+2a}{2}$, $ \mathrm{mult}_{E_1}(\mu A)\leq \frac{3+4a_1+2s_A-a_2-a_3-a_4+4a}{4}$ and $\mathrm{mult}_{E_1}(\mu A)\leq\frac{2+3a_1+s_A+3a}{3}$ by the similar way   of Lemma \ref{s23}, \ref{s234} and \ref{sR}, respectively.
And  effective divisors which is a sum of divisors $a(l_8+E_8)$ and an effective divisors constructed in  Lemma \ref{s23}, \ref{s234} and \ref{sR} give upper bound of $\alpha_c(S,\mu A)$.

When $s_A\leq 1$,
write $D=\alpha E_1+bZ+cl_8+\Omega$, where the support of $\Omega$ does not contain $E_1, l_8$ and $Z$, then we have
\begin{eqnarray*}
2+4a_1+2s_A+4a=(Z+l_8)\cdot D\geq 4\alpha.
\end{eqnarray*}
There is an effective divisor
 \begin{eqnarray*}
\mu A\sim_{\mathbb{Q}}  \frac{1-s_A}{2}Z+\sum_{i=2}^8a_iC_i+\frac{1+2a_1+s_A+2a}{2}E_1+al_8.
\end{eqnarray*}

Finally we obtain the following statement.

\begin{proposition}\label{alphaCcobdleF}
Let $S$ be a smooth del Pezzo surface and $A$ be a ample divisor of $S$.  If the morphism $\phi$ is a conic bundle, $\mu A\sim_{\mathbb{Q}} -K_{S}+aB+\sum_{i=1}^7a_iE_i $ and it  factors through  the Hirzebruch surface $\mathbb{F}_1$,  then

\begin{itemize}
\item[$\bullet$]When $s_A>4$,  $\alpha_c(S,A)=\frac{1}{2+a_1+a}$;
\item[$\bullet$]When $4\geq s_A>1$,\\ $\alpha_c(S,A)=\mathrm{Max}\{\frac{2}{2+2a_1+s_A-a_2-a_3+2a},\frac{4}{3+4a_1+2s_A-a_2-a_3-a_4+4a},\frac{3}{2+3a_1+s_A+3a}\}$;
\item[$\bullet$]When $1\geq s_A$,  $\alpha_c(S,A)= \mathrm{min}\{\frac{2}{1+2a_1+s_A+2a},1\}$.
\end{itemize}


\end{proposition}
\bigskip

\textbf{Subcase 2:} $R$ is isomorphic to $\mathbb{P}^1\times\mathbb{P}^1$.

Use the notation $s_{Am}=s_A-a_m$ so that $s_{A7}\geq \cdots \geq s_{A1}$.
The present subsection shows

\begin{proposition}\label{alphaCcobdleP}
Let $S$ be a smooth del Pezzo surface and $A$ be a ample divisor of $S$.  If the morphism $\phi$ is a conic bundle, $\mu A\sim_{\mathbb{Q}} -K_{S}+aB+\sum_{i=1}^7a_iE_i $ and it  factors through   $\mathbb{P}^1\times\mathbb{P}^1$,  then

\begin{itemize}
\item[$\bullet$]When $s_A>4$,  $\alpha_c(S,A)=\frac{1}{2+a_1+a}$;
\item[$\bullet$]When $4\geq s_A>1$, $\alpha_c(S,A)=\mathrm{Max}\{\frac{2}{2+s_{A7}-a_2-a_3+2a},\frac{4}{3+2s_{A7}-a_2-a_3-a_4+4a},\frac{3}{2+s_{A7}+3a}\}$;
\item[$\bullet$]When $1\geq s_A$,  $\alpha_c(S,A)= \mathrm{min}\{\frac{2}{1+s_{A7}+2a},1\}$.
\end{itemize}


\end{proposition}
For each $-1$-curve $E_m$, we have another $-1$-curve $E^a_m$ in the fiber of $\phi$ that contains $E_i$. In addition, there is a $-1$-curve $E^b_m$ that intersects $E^b_m$ but none of the other exceptional divisors of $\phi_1$. The curve $B$ is linearly equivalent to the divisor $E_m+E^a_m$.   Then for each $m$ use different labelling $\mathcal{L}_m$ as follows:
 \begin{eqnarray*}
(a'_1,\ldots,a'_8)=(a_1,\ldots,\hat{a}_m,\ldots, a_7,0, 0),\\
(E'_1,\ldots,E'_8)=(E_1,\ldots, \hat{E}_m, \ldots , E_7,E^a_m,E^b_m),
\end{eqnarray*}
where the sequences  $(a_1,\ldots,\hat{a}_m,\ldots, a_7,0, 0)$ and $(E_1,\ldots, \hat{E}_m, \ldots , E_7,E^a_m,E^b_m)$ are obtained by deleting $a_m$ and $E_m$ from sequences $(a_1,\ldots, a_7,0, 0)$ and $(E_1,\ldots, E_7,E^a_m,E^b_m)$, respectively.
By substituting $(a_1,\ldots,a_8)$ and $(E_1,\ldots,E_8)$ by $(a_1,\ldots,a_8)$ and $(E'_1,\ldots,E'_8)$ and applying same arguments in the case of \emph{Subcase 1}, we obtain the same results as \emph{Subcase 1} in terms of  $(a'_1,\ldots,a'_8)$.
 It yields to a required statement since $s_{A7}\geq \cdots \geq s_{A1}$.


\section{Proof of Main theorem}
Let $\mathrm{LCS}(S,D) \subset S$ be the subset such that $P \in \mathrm{LCS}(S,D)$ if and only if
$(S,D)$ is not log terminal at the point $P$. The set $\mathrm{LCS}(S,D)$ is called the
locus of log canonical singularities.
\begin{lemma}\label{lcs}
Suppose that $-(K_S +D)$ is ample. Then the set $\mathrm{LCS}(S,D)$
is connected.
\end{lemma}
\begin{proof}
See Theorem 17.4 in \cite{Ko}.
\end{proof}

\begin{lemma}\label{ptalpha}
Let $S$ be a smooth del Pezzo surfaces of degree one and $A$ be an ample divisor on $S$. Either the alpha invariant $\alpha(S,\mu A)$ is greater than $\frac{2}{3+a_1}$ or $\alpha(S,\mu A)=\alpha_c(S,\mu A)$.
\end{lemma}

\begin{proof}
 Let $ \lambda <\frac{2}{3+a_1}$. Suppose that the pair $(S, \lambda \mu A)$ is not log canonical at the point $p$ on $S$ and log canonical along a punctured neighborhood of $p$.    We will use  notations in the section 3.

When $\phi$ is birational, $E_1,\ldots ,E_8$ are mutually disjoint. Suppose the point $p$ does not lies on  the curves $E_1,\ldots ,E_8$. Then the contraction of $E_i$'s, $\pi:S\rightarrow \mathbb{P}^2$ is a locally isomorphism in a neighborhood of $p$ and $\pi(\mu A)=-K_{\mathbb{P}^2}$. Take a general line $L$ in $\mathbb{P}^2$, then $-K_{\mathbb{P}^2}-(\lambda \mu \pi(A)+L)$ is ample. But $\mathrm{LCS}(\mathbb{P}^2,\lambda \mu A+L)$ contains disjoint set $\pi(p)\cup L$. It is contradiction by Lemma \ref{lcs}.

Now suppose the point $p$ lies on $E_i$. Let $l_{ij}$ be a unique (-1)-curve meeting $E_i$ and $E_j$. Choose $-1$-curves $l_{ij},l_{jk},l_{ik}$ such that these curves do not pass through $p$.
Then the  curves $\{l_{ij},l_{jk},l_{ik}\}\cup(\{E_1,\ldots,E_8\}\setminus\{E_i,E_j,E_k\})$
are mutually disjoint so that we obtain $\pi':S\rightarrow \mathbb{P}^2$ by contracting  these curves. Then  we have that $\mu \pi'(A)= -K_{\mathbb{P}^2}+a_i\pi'(E_i)+a_j\pi'(E_j)+a_k\pi'(E_k)$ and $-K_{\mathbb{P}^2}-(\lambda \mu \pi'(A)+(1-\lambda a_j) \pi'(E_j)-\lambda a_k \pi'(E_k))$ is ample. But $\mathrm{LCS}(\mathbb{P}^2,\lambda \mu A+(1-\lambda a_j) \pi'(E_j)-\lambda a_k \pi'(E_k))$ contains disjoint set $\pi'(p)\cup \pi'(E_j)$. It is contradiction by Lemma \ref{lcs}.

  When $\phi$ is conic bundle which factors through $\mathbb{F}_1$, $\mu A=-K_S+\sum_{i=1}^7 a_i E_i+aB $, the contraction of $E_1,\ldots,E_7$ defines  the morphism $S\rightarrow \mathbb{F}_1$
  and there is a unique $-1$-curve $E_8$ in $\mathbb{F}_1$.
  The curves $E_1,\ldots ,E_8$ are mutually disjoint. Suppose the point $p$ does not lie on  the curves $E_1,\ldots ,E_8$. As the case that $\phi$ is birational, the contraction of $E_i$'s, $\pi:S\rightarrow \mathbb{P}^2$ is a locally isomorphism in a neighborhood of $p$.   The divisor $-K_{\mathbb{P}^2}-(\lambda \mu \pi(A)+(1-\lambda a)\pi(B))$ is ample. But $\mathrm{LCS}(\mathbb{P}^2,\lambda \mu \pi(A)+(1-\lambda a)\pi(B))$ contains disjoint set $\pi(p)\cup \pi(B)$. It is a contradiction to Lemma \ref{lcs}.
   If the point $p$ lies on $E_i$, let $\pi':S\rightarrow \mathbb{P}^2$ be a same contraction in the case that $\phi$ is birational.  Taking $B\sim_{\mathbb{Q}} E_j+E_j'$, where $E_j'$ is $-1$-curve which meets $E_j$ and none of other $E_i$'s.
   A divisor $-K_{\mathbb{P}^2}-(\lambda \mu \pi'(A)+(1-\lambda (a_j+a)) \pi'(E_j)-\lambda a_k \pi'(E_k)-a\pi'(E_i'))$ is ample. But $\mathrm{LCS}(\mathbb{P}^2,\lambda \mu A+(1-\lambda a_j) \pi'(E_j)-\lambda a_k \pi'(E_k)-a\pi'(E_i'))$ contains disjoint set $\pi'(p)\cup \pi'(E_j)$. It is a contradiction to Lemma \ref{lcs}.

When $\phi$ is conic bundle which factors through $\mathbb{P}^1\times\mathbb{P}^1$,  the contraction of $E_1,\ldots, E_6$ defines the morphism $S\rightarrow S_7$. The del Pezzo surface $S_7$ has two disjoint $-1$-curves $E_7^a,E_7^b$ which intersects $E_7$. The curves $E_1,\ldots ,E_6,E_7^a,E_7^b$ are mutually disjoint. For this disjoint set of curves, similar arguments as the cases of conic bundle which factors through $\mathbb{F}_1$ force to a contradiction.

\end{proof}

\emph{Proof of Main Theorem \ref{mtheorem}}. We can write  $\mu A =-K_S +\sum_{i=1}^8 a_iE_i+aB$, where $a=0$ if $\phi $ is birational and $a_8=0$ if  $\phi $ is conic bundle.

Note that $\mu \nu(\mu A)=\nu (A)$ and $\mu \alpha(S,\mu A)=\alpha(S,A)$. Therefore $\alpha(S,A)>\frac{2}{3}\nu(A)$ if and only if $\alpha(S,\mu A)>\frac{2}{3}\nu(\mu A)$.
Suppose that $-K_S-\frac{2}{3}\nu (A) A$ is nef. Then $-K_S-\frac{2}{3} \nu(\mu A) \mu A$  is also nef. Thus the intersection number $(-K_S-\frac{2}{3} \nu(\mu A) \mu A)\cdot (6h-3e_1-2e_2-2e_3-2e_4-2e_5-2e_6-2e_7-2e_8)$ is non-negative. It means that $$\frac{1}{3a_1+2s_A+1+2a}\geq \frac{2}{3}\nu (\mu A).$$

By Lemma \ref{ptalpha}, either $\alpha(S,\mu A)>\frac{2}{3+a_1}$ or  $\alpha(S,\mu A)=\alpha_c(S,\mu A)$. Suppose, the latter then, we can easily check that all the values of $\alpha_c(S,\mu A)$ in propositions \ref{alphaCbirational}, \ref{alphaCcobdleF}, \ref{alphaCcobdleP} are greater than the value $\frac{1}{3a_1+2s_A+1+2a}$ according to the range of $s_A$. It means that $\alpha(S,\mu A)=\alpha_c(S,\mu A)>\frac{2}{3}\nu(\mu A)$. Suppose $\alpha(S,\mu A)\neq\alpha_c(S,\mu A)$, then $\alpha(S,\mu A)>\frac{2}{3+a_1}$.
If $5a_1+4s_A+4a>1$, then
$\frac{2}{3+a_1}>\frac{1}{3a_1+2s_A+1+2a}\geq \frac{2}{3}\nu(\mu A)$. On the other hand, assume that  $2a_1+a\leq 5a_1+4s_A+4a\leq 1$.
Use the fact that $0\leq a_i\leq a_i(2-a_i)$ for any $a_i$, and for positive numbers $A,B$ and non-negative numbers $C$ and $D$, if $A\leq B$ and $C\leq D$, $\frac{A}{B}\geq \frac{A+C}{B+D}$.
Then it is easily verified that $$\frac{2}{3+a_1}>\frac{2}{3}\frac{1+a_1+2a}{1+2a_1-a_1^2+2a}\geq \frac{2}{3}\frac{1+\sum_{i=1}^8 a_i +2a }{1+\sum_{i=1}^8 a_i(2-a_i)+2a}=\frac{2}{3}\frac{-K_S\cdot (\mu A) }{(\mu A)^2}=\frac{2}{3}\nu(\mu A).$$

Thus in any case, $\alpha(S,\mu A)>\frac{2}{3}\nu (\mu A)$.

\qed
\begin{conjecture}
If $S$ be a smooth del Pezzo surface of degree 1 which has no cuspidal curve in $|-K_S|$, then we have $\alpha(S,A)=\alpha_c(S,A)$ for any ample divisor $A$ in $S$.
\end{conjecture}

\end{document}